\documentclass[11pt,reqno]{amsart}

\usepackage[T1]{fontenc}
\usepackage{lmodern}
\usepackage{microtype}
\usepackage{amsmath,amssymb,amsthm,mathtools,mathrsfs}
\usepackage{bm}
\usepackage{enumitem}
\usepackage{booktabs}
\usepackage{array}
\usepackage{xcolor}
\usepackage[colorlinks=true,linkcolor=blue!55!black,citecolor=blue!55!black,urlcolor=blue!55!black]{hyperref}
\usepackage[nameinlink,capitalise,noabbrev]{cleveref}
\usepackage[a4paper,margin=1.08in]{geometry}

\numberwithin{equation}{section}

\newtheorem{theorem}{Theorem}[section]
\newtheorem{proposition}[theorem]{Proposition}
\newtheorem{lemma}[theorem]{Lemma}

\theoremstyle{definition}

\theoremstyle{remark}

\crefname{theorem}{theorem}{theorems}
\Crefname{theorem}{Theorem}{Theorems}
\crefname{proposition}{proposition}{propositions}
\Crefname{proposition}{Proposition}{Propositions}
\crefname{lemma}{lemma}{lemmas}
\Crefname{lemma}{Lemma}{Lemmas}
\crefname{corollary}{corollary}{corollaries}
\Crefname{corollary}{Corollary}{Corollaries}
\crefname{remark}{remark}{remarks}
\Crefname{remark}{Remark}{Remarks}

\DeclareMathOperator{\Per}{Per}
\DeclareMathOperator{\Tr}{Tr}
\DeclareMathOperator{\diag}{diag}

\newcommand{\E}{\mathbb E}
\newcommand{\Pp}{\mathbb P}
\newcommand{\R}{\mathbb R}
\newcommand{\C}{\mathbb C}
\newcommand{\1}{\mathbf 1}
\newcommand{\dd}{\mathrm d}
\newcommand{\ii}{\mathrm i}

\newcommand{\weak}{\xrightarrow{\mathrm w}}

\newcommand{\Sym}{\mathrm{Sym}}
\newcommand{\Herm}{\mathrm{Herm}}

\title[Permanental roots of Gaussian random matrices]{Rotated semicircle laws for permanental roots of Gaussian random matrices}
\author{Renjie Feng}
\email{renjie.feng@sydney.edu.au}
\address{Sydney Mathematical Research Institute, University of Sydney}

\author{Dong Yao}
\email{dongyao@jsnu.edu.cn}
\address{School of Mathematics and Statistics/RIMS, Jiangsu Normal University, China}
\date{\today}

\begin{document}

\begin{abstract}
For matrices drawn from the standard Gaussian orthogonal ensemble (GOE) and
Gaussian unitary ensemble (GUE), we prove that the normalized zero counting
measure of the permanental characteristic polynomial
\(\Per(zI_N-H_N)\) converges almost surely to the standard Wigner semicircle
law on \([-2,2]\), rotated by \(\pi/2\) onto the imaginary axis.  This proves
the conjecture proposed by Fyodorov in \cite{Fyodorov2006}.
\end{abstract}

\maketitle

\section{Introduction}

For an \(N\times N\) complex matrix \(A=(a_{jk})\), its permanent is
\begin{equation}\label{eq:def-permanent}
  \Per A
  :=
  \sum_{\sigma\in S_N}\prod_{j=1}^N a_{j,\sigma(j)},
\end{equation}
and its permanental characteristic polynomial is
\[
  P_A(z):=\Per(zI_N-A).
\]
The polynomial \(P_A\) is monic of degree \(N\).  Its zeros, counted with
algebraic multiplicity, will be called the \emph{permanental roots} of \(A\).

The permanent connects combinatorics, graph theory, statistical mechanics,
and quantum physics.  If \(A\) is the \(0\)--\(1\) biadjacency matrix of a
bipartite graph, then \(\Per A\) counts its perfect matchings; for a
nonnegative weighted matrix, it gives the corresponding weighted matching
partition function.  For noninteracting identical particles, symmetrization
of multiparticle amplitudes leads to permanents for bosons, whereas
antisymmetrization leads to determinants for fermions.  In particular,
multiphoton transition amplitudes in passive linear-optical networks are
expressed through matrix permanents \cite{Scheel2004}.  This connection
underlies the BosonSampling model, which uses the computational difficulty of
permanents to construct sampling problems expected to be hard for classical
computers \cite{AaronsonArkhipov2013}.

In computational complexity theory, the permanent exhibits a striking
contrast with the determinant.  Determinants can be computed in polynomial
time by Gaussian elimination, whereas Valiant proved that computing the
permanent of a \(0\)--\(1\) matrix is \(\#\mathrm P\)-complete
\cite{Valiant1979}.  The permanent is therefore a canonical complete problem
in exact counting complexity.  The classical formulas of Ryser and Glynn
admit implementations using \(O(N2^N)\) arithmetic operations and hence remain
exponential in \(N\) \cite{Ryser1963,Glynn2010}.  For approximation, Jerrum,
Sinclair, and Vigoda constructed a fully polynomial-time randomized
approximation scheme for the permanent of a matrix with nonnegative entries
\cite{JerrumSinclairVigoda2004}.  Their theorem does not directly extend to
arbitrary signed or complex matrices, where cancellation between different
terms creates an additional difficulty for relative-error approximation.

Fyodorov \cite{Fyodorov2006} initiated the study of permanental
characteristic polynomials of random matrices and conjectured that, for the
GOE and GUE, the permanental roots have a global limiting density on the
imaginary axis given by a rotated semicircle law.  The present paper proves
this conjecture.

For \(\beta\in\{1,2\}\), let
\[
  \mathcal H_N^{(1)}=\Sym_N(\R),
  \qquad
  \mathcal H_N^{(2)}=\Herm_N(\C).
\]
We regard these spaces as finite-dimensional real Euclidean vector spaces and
write \(\dd H\) for their Lebesgue measures.  The standard Gaussian ensemble
on \(\mathcal H_N^{(\beta)}\) is the probability measure
\[
  \dd\Pp_{\beta,N}(H)
  =
  \frac{1}{\mathcal Z_{\beta,N}}
  \exp\!\left(-\frac{\beta N}{4}\Tr H^2\right)\dd H,
\]
where \(\mathcal Z_{\beta,N}\) is the normalizing constant.  For \(\beta=1\),
this is the GOE: the diagonal entries are independent \(N(0,2/N)\) random
variables, and the strict upper-triangular entries are independent
\(N(0,1/N)\) random variables.  For \(\beta=2\), this is the GUE: the
diagonal entries are independent \(N(0,1/N)\) random variables, while the
real and imaginary parts of each strict upper-triangular entry are independent
\(N(0,1/(2N))\) random variables; see
\cite[Sections~2.1, 2.2 and~2.5]{AndersonGuionnetZeitouni2010}.
Under this normalization, the empirical eigenvalue measures converge weakly
in probability to the
standard semicircle law \cite[Theorem~2.1.1]{AndersonGuionnetZeitouni2010},
\[
  \dd\mu_{\mathrm{sc}}(x)
  =
  \frac{1}{2\pi}\sqrt{4-x^2}\,
  \1_{[-2,2]}(x)\dd x.
\]

Fix \(\beta\in\{1,2\}\).  Let \((H_N)_{N\ge1}\) be a sequence of random
matrices defined on a common probability space such that
\(H_N\sim\Pp_{\beta,N}\) for every \(N\ge1\).   Define
\[
  P_N(z):=\Per(zI_N-H_N),
\]
and let \(\zeta_{1,N},\ldots,\zeta_{N,N}\) be its zeros, counted with
algebraic multiplicity.  The normalized zero counting measure is
\[
  \nu_{P_N}:=\frac1N\sum_{j=1}^N\delta_{\zeta_{j,N}},
\]
where \(\delta_\zeta\) denotes the Dirac probability measure at
\(\zeta\in\C\).

For a measurable map \(T:X\to Y\) and a Borel measure \(\mu\) on \(X\), we
write \(T_\#\mu\) for the push-forward measure on \(Y\), defined by
\[
  (T_\#\mu)(B):=\mu\bigl(T^{-1}(B)\bigr)
\]
for every Borel set \(B\subseteq Y\).  Let
\[
  \mathsf m_i:\C\to\C,
  \qquad
  \mathsf m_i(z)=\ii z.
\]
We write \(\mu_N\weak\mu\) for weak convergence of Borel probability
measures.

\begin{theorem}\label{thm:main}
With the notation above,
\[
  \nu_{P_N}
  \weak
  (\mathsf m_i)_\#\mu_{\mathrm{sc}}
  \qquad\text{almost surely.}
\]
Equivalently, for every bounded continuous function \(f:\C\to\C\),
\[
  \frac1N\sum_{j=1}^N f(\zeta_{j,N})
  \to
  \int_{-2}^{2}
  f(\ii x)\frac1{2\pi}\sqrt{4-x^2}\,\dd x
  \qquad\text{almost surely.}
\]
\end{theorem}

\Cref{thm:main} determines the global limiting distribution of the
permanental roots, but it leaves several problems open.  Fyodorov's
original argument for the GUE was already based on an expected
logarithmic-potential approach.  He introduced
\[
  \Phi(x,y)
  :=
  \lim_{N\to\infty}\frac1N
  \E\log\bigl|P_N(x+\ii y)\bigr|^2
\]
and observed that the rotated semicircle law would follow if one could
justify the interchange
\cite[Eqs.~(1.17)--(1.20)]{Fyodorov2006}
\[
  \lim_{N\to\infty}\frac1N\E\log|P_N(z)|^2
  =
  \lim_{N\to\infty}\frac1N\log\E|P_N(z)|^2.
\]
 A natural strengthening is to ask whether the random logarithmic potentials
satisfy
\begin{equation}\label{eq:future-log-potential}
  \frac1N\log|P_N(z)|
  \to
  \int_{-2}^{2}\log|z-\ii x|\,\dd\mu_{\mathrm{sc}}(x)
\end{equation}
in probability or almost surely in
\(L^1_{\mathrm{loc}}(\C)\).   By the
Poincar\'e--Lelong identity,
\[
  \nu_{P_N}
  =
  \frac{1}{2\pi N}\Delta\log|P_N|
\]
in the sense of distributions. Thus convergence in \(L^1_{\mathrm{loc}}(\C)\)  would imply the corresponding convergence
of zero counting measures directly from \eqref{eq:future-log-potential}.

A related question concerns quantitative confinement of the roots to the
imaginary axis and to the limiting interval.  
It would therefore be interesting to estimate separately the transverse
displacement
\[
  \max_{1\le j\le N}|\Re\zeta_{j,N}|
\]
and the possible overshoot beyond the limiting edges,
\[
  \max_{1\le j\le N}
  \bigl(|\Im\zeta_{j,N}|-2\bigr)_+,
  \qquad t_+:=\max\{t,0\}.
\]

One may also investigate the local statistics of the permanental roots, in
analogy with the local eigenvalue theory of Gaussian random matrices.  After
localization near the imaginary axis and the appropriate microscopic
rescaling, it is natural to ask whether bulk correlation functions exhibit
sine-kernel universality.  Near the limiting edges \(\pm2\ii\), one may
similarly ask whether the rescaled edge point processes are governed by
Airy-type kernels.  As in the classical GOE and GUE theory, the precise
limiting structures may be Pfaffian in the orthogonal case and determinantal
in the unitary case.

Closely related to the edge point process is the fluctuation of the extremal
permanental roots.  In particular, one may ask whether
\[
  \max_{1\le j\le N}\Im\zeta_{j,N}
\]
has, after the appropriate centering and scaling, a Tracy--Widom-type limit,
with possibly different limiting distributions in the GOE and GUE cases.  We
refer to \cite{AndersonGuionnetZeitouni2010} for the classical sine-kernel
bulk limits, Airy-kernel edge limits, and Tracy--Widom laws for Gaussian
random matrices.

Finally, the proof given here uses exact identities that are special to
Gaussian ensembles.  Establishing the same global rotated semicircle law for
non-Gaussian real symmetric and complex Hermitian Wigner matrices would
constitute a natural universality problem.

\section{Permanent identities}

For \(S\subseteq\{1,\dots,N\}\), let \(A[S]\) denote the principal submatrix
whose rows and columns are indexed by \(S\), and put \(\Per A[\varnothing]=1\).

\begin{lemma}
For every complex \(N\times N\) matrix \(A\),
\begin{equation}\label{eq:principal-expansion}
  \Per(zI_N-A)
  =
  \sum_{S\subseteq\{1,\dots,N\}}
  (-1)^{|S|}\Per A[S]z^{N-|S|}.
\end{equation}
In particular, \(\Per(zI_N-A)\) is monic of degree \(N\).
\end{lemma}

\begin{proof}
Expand the permanent using \eqref{eq:def-permanent}.  In a term indexed by
\(\sigma\in S_N\), choose the matrix entry rather than the factor \(z\)
exactly on a subset \(S\).  If \(j\notin S\), then necessarily
\(\sigma(j)=j\), so \(\sigma\) restricts to a permutation of \(S\).  Summing
over all restrictions gives \((-1)^{|S|}\Per A[S]\), while the indices
outside \(S\) contribute \(z^{N-|S|}\).
\end{proof}

\begin{lemma}\label{lem:real-coefficients}
If \(A\) is Hermitian, then \(\Per(zI_N-A)\) has real coefficients.
\end{lemma}

\begin{proof}
Every principal submatrix \(A[S]\) is Hermitian.  Since the permanent is
invariant under transpose,
\[
  \overline{\Per A[S]}
  =
  \Per\overline{A[S]}
  =
  \Per A[S]^{\mathsf T}
  =
  \Per A[S].
\]
Now apply \eqref{eq:principal-expansion}.
\end{proof}

For
\[
  q=
  \begin{pmatrix}
    q_{11}&q_{12}\\
    \overline q_{12}&q_{22}
  \end{pmatrix}
  \in\Herm_2(\C),
\]
where \(q_{11},q_{22}\in\R\) and \(q_{12}\in\C\), set
\[
  \dd q
  :=
  \dd q_{11}\,\dd q_{22}\,
  \dd(\Re q_{12})\,\dd(\Im q_{12})
\]
and define
\begin{equation}\label{eq:aux-q-law}
  \dd\gamma_{q,N}(q)
  :=
  \frac{N^2}{2\pi^2}
  \exp\!\left(-\frac N2\Tr q^2\right)\dd q.
\end{equation}
For \(\eta\in\C\), set
\[
  \dd^2\eta:=\dd(\Re\eta)\,\dd(\Im\eta)
\]
and define
\begin{equation}\label{eq:eta-law}
  \dd\gamma_{\eta,N}(\eta)
  :=
  \frac N\pi e^{-N|\eta|^2}\dd^2\eta.
\end{equation}
Since
\[
  \Tr q^2=q_{11}^2+q_{22}^2+2|q_{12}|^2,
\]
elementary Gaussian integration shows that \(\gamma_{q,N}\) and
\(\gamma_{\eta,N}\) are probability measures.

Let
\[
  M:=\diag(\mu_1,\mu_2).
\]
The following identity is the \(n=2\) case of Fyodorov's general GUE
\(n\)-point duality \cite[Theorem~1.2 and Eq.~(4.17)]{Fyodorov2006}.

\begin{proposition} 
Let \(\mu_1,\mu_2\in\C\).  If \(H_N\sim\Pp_{2,N}\), then
\begin{equation}\label{eq:GUE-per-duality}
  \E[P_N(\mu_1)P_N(\mu_2)]
  =
  \int_{\Herm_2(\C)}
  [\Per(M-q)]^N\,\dd\gamma_{q,N}(q).
\end{equation}
\end{proposition}

Fyodorov also obtained the corresponding GOE two-point duality
\cite[Eqs.~(7.7)--(7.9)]{Fyodorov2006}.   
\begin{proposition} 
Let \(\mu_1,\mu_2\in\C\).  If \(H_N\sim\Pp_{1,N}\), then
\begin{equation}\label{eq:GOE-per-duality}
  \E[P_N(\mu_1)P_N(\mu_2)]
  =
  \int_{\Herm_2(\C)}\int_{\C}
  \bigl(\Per(M-q)+|\eta|^2\bigr)^N
  \,\dd\gamma_{\eta,N}(\eta)\,\dd\gamma_{q,N}(q).
\end{equation}
\end{proposition}

We next specialize these formulas to \(\mu_1=\ii x\) and
\(\mu_2=-\ii x\), and rewrite the resulting expressions as two-dimensional
scalar Gaussian integrals.

For \(\alpha>-1\), define
\[
  \dd\Gamma_{\alpha,N}(y_1,y_2)
  :=
  \frac{1}{Z_{\alpha,N}}
  |y_1-y_2|^\alpha
  \exp\!\left[-\frac N2(y_1^2+y_2^2)\right]
  \dd y_1\,\dd y_2,
\]
where
\[
  Z_{\alpha,N}
  :=
  \int_{\R^2}|y_1-y_2|^\alpha
  \exp\!\left[-\frac N2(y_1^2+y_2^2)\right]
  \dd y_1\,\dd y_2.
\]

\begin{lemma}\label{lem:Z-alpha}
For every \(\alpha>-1\),
\[
  Z_{\alpha,N}
  =
  2^{\alpha+1}\sqrt\pi\,
  \Gamma\!\left(\frac{\alpha+1}{2}\right)
  N^{-(\alpha+2)/2}.
\]
In particular,
\[
  Z_{2,N}=\frac{4\pi}{N^2},
  \qquad
  Z_{4,N}=\frac{24\pi}{N^3}.
\]
\end{lemma}

\begin{proof}
Set
\[
  s=\frac{y_1+y_2}{2},
  \qquad
  r=\frac{y_1-y_2}{2}.
\]
Then
\[
  \dd y_1\,\dd y_2=2\,\dd s\,\dd r,
  \qquad
  |y_1-y_2|^\alpha=2^\alpha|r|^\alpha,
  \qquad
  \frac12(y_1^2+y_2^2)=s^2+r^2.
\]
It follows that
\begin{align*}
  Z_{\alpha,N}
  &=
  2^{\alpha+1}
  \left(\int_{\R}e^{-Ns^2}\,\dd s\right)
  \left(\int_{\R}|r|^\alpha e^{-Nr^2}\,\dd r\right)\\
  &=
  2^{\alpha+1}\sqrt\pi\,
  \Gamma\!\left(\frac{\alpha+1}{2}\right)
  N^{-(\alpha+2)/2}.
\end{align*}
The special cases follow from
\(
  \Gamma(3/2)=\sqrt\pi/2
\)
and
\(
  \Gamma(5/2)=3\sqrt\pi/4.
\)
\end{proof}

\begin{proposition}\label{prop:scalar-permanent-duality}
For every \(x\in\R\),
\begin{align}
  H_N\sim\Pp_{2,N}:\qquad
  &\E[P_N(\ii x)P_N(-\ii x)]
  =
  \int_{\R^2}
  \prod_{\ell=1}^2(x-\ii y_\ell)^N
  \,\dd\Gamma_{2,N}(y_1,y_2),
  \label{eq:GUE-per-scalar}\\
  H_N\sim\Pp_{1,N}:\qquad
  &\E[P_N(\ii x)P_N(-\ii x)]
  =
  \int_{\R^2}
  \prod_{\ell=1}^2(x-\ii y_\ell)^N
  \,\dd\Gamma_{4,N}(y_1,y_2).
  \label{eq:GOE-per-scalar}
\end{align}
The integrals are absolutely convergent.
\end{proposition}

\begin{proof}
We first consider the GUE.  In \eqref{eq:GUE-per-duality}, take
\[
  M=\diag(\ii x,-\ii x).
\]
Make the change of variable \(q_{11}\mapsto-q_{11}\), which preserves the
measure \(\gamma_{q,N}\), and then relabel the reflected coordinate as
\(q_{11}\).  A direct calculation gives
\[
  \Per(M-q)
  =
  x^2-\ii x(q_{11}+q_{22})-q_{11}q_{22}+|q_{12}|^2.
\]
Write
\[
  s=\frac{q_{11}+q_{22}}2,
  \qquad
  v=\frac{q_{11}-q_{22}}2,
  \qquad
  u=\Re q_{12},
  \qquad
  \omega=\Im q_{12},
\]
and put
\[
  r=(v^2+u^2+\omega^2)^{1/2}.
\]
Then
\[
  \Per(M-q)
  =
  (x-\ii(s+r))(x-\ii(s-r)),
  \qquad
  \Tr q^2=2(s^2+r^2).
\]
The absolute Jacobian of the linear transformation
\((s,v)\mapsto(q_{11},q_{22})\) is \(2\).  Hence
\eqref{eq:aux-q-law} becomes
\[
  \frac{N^2}{\pi^2}e^{-N(s^2+r^2)}
  \,\dd s\,\dd v\,\dd u\,\dd\omega.
\]
Using polar coordinates in \(\R^3\), together with \(|S^2|=4\pi\), gives
\begin{equation}\label{eq:GUE-s-r-density}
  \frac{4N^2}{\pi}r^2e^{-N(s^2+r^2)}\,\dd s\,\dd r,
  \qquad s\in\R,\quad r\ge0.
\end{equation}
Now set
\[
  y_1=s+r,
  \qquad
  y_2=s-r.
\]
This maps \(\R\times[0,\infty)\) onto the half-plane
\(\{(y_1,y_2)\in\R^2:y_1\ge y_2\}\), up to boundary sets of Lebesgue
measure zero.  Moreover,
\[
  \dd s\,\dd r=\frac12\,\dd y_1\,\dd y_2,
  \qquad
  r=\frac{y_1-y_2}{2}.
\]
Therefore the push-forward of \eqref{eq:GUE-s-r-density} has density
\begin{equation}\label{eq:ordered-beta2-density}
  \frac{N^2}{2\pi}(y_1-y_2)^2
  e^{-\frac N2(y_1^2+y_2^2)}\,\dd y_1\,\dd y_2
\end{equation}
on \(y_1\ge y_2\).  The function
\[
  (y_1,y_2)\longmapsto
  (x-\ii y_1)^N(x-\ii y_2)^N
\]
is symmetric in \(y_1,y_2\).  Since \(Z_{2,N}=4\pi/N^2\), integrating
this symmetric function against \eqref{eq:ordered-beta2-density} on
\(y_1\ge y_2\) is equal to integrating it against \(\Gamma_{2,N}\) on
all of \(\R^2\).  This proves \eqref{eq:GUE-per-scalar}.

We next consider the GOE.  Starting from \eqref{eq:GOE-per-duality}, make
the same reflection \(q_{11}\mapsto-q_{11}\), and write
\[
  \eta=\eta_1+\ii\eta_2.
\]
With \(s,v,u,\omega\) as above, set
\[
  r=
  \bigl(v^2+u^2+\omega^2+\eta_1^2+\eta_2^2\bigr)^{1/2}.
\]
Then
\[
  \Per(M-q)+|\eta|^2
  =
  (x-\ii(s+r))(x-\ii(s-r)).
\]
Combining \eqref{eq:aux-q-law} and \eqref{eq:eta-law}, and including the
Jacobian \(2\) from \((s,v)\mapsto(q_{11},q_{22})\), gives
\[
  \frac{N^3}{\pi^3}e^{-N(s^2+r^2)}
  \,\dd s\,\dd v\,\dd u\,\dd\omega\,\dd\eta_1\,\dd\eta_2.
\]
Using polar coordinates in \(\R^5\), together with
\(|S^4|=8\pi^2/3\), gives
\begin{equation}\label{eq:GOE-s-r-density}
  \frac{8N^3}{3\pi}r^4e^{-N(s^2+r^2)}\,\dd s\,\dd r,
  \qquad s\in\R,\quad r\ge0.
\end{equation}
Under the same transformation \(y_1=s+r\), \(y_2=s-r\), the push-forward
of \eqref{eq:GOE-s-r-density} has density
\begin{equation}\label{eq:ordered-beta4-density}
  \frac{N^3}{12\pi}(y_1-y_2)^4
  e^{-\frac N2(y_1^2+y_2^2)}\,\dd y_1\,\dd y_2
\end{equation}
on the half-plane \(y_1\ge y_2\).  Since \(Z_{4,N}=24\pi/N^3\) and the
integrand is symmetric in \(y_1,y_2\), integration against
\eqref{eq:ordered-beta4-density} on \(y_1\ge y_2\) equals integration
against \(\Gamma_{4,N}\) on all of \(\R^2\).  This proves
\eqref{eq:GOE-per-scalar}.

\end{proof}

\section{Gaussian \texorpdfstring{$L^2$}{L2} norms}

By \Cref{lem:Z-alpha}, it is convenient to set
\begin{equation}\label{eq:Z-alpha-base}
  Z_\alpha^\circ
  :=
  Z_{\alpha,2}
  =
  2^{\alpha/2}\sqrt\pi\,
  \Gamma\!\left(\frac{\alpha+1}{2}\right),
  \qquad \alpha>-1.
\end{equation}
Then, for every \(N\ge1\),
\begin{equation}\label{eq:Z-alpha-N-value}
  Z_{\alpha,N}
  =
  \left(\frac2N\right)^{1+\alpha/2}Z_\alpha^\circ.
\end{equation}
For \(u>0\) and \(m\in\mathbb N_0\), we use the rising factorial
(or Pochhammer symbol)
\[
  (u)_0:=1,
  \qquad
  (u)_m:=u(u+1)\cdots(u+m-1)
  =\frac{\Gamma(u+m)}{\Gamma(u)}
  \quad(m\ge1).
\]
For \(m\ge0\), define
\begin{equation}\label{eq:J-alpha-m}
  J_m^{(\alpha)}
  :=
  \int_{\R^3}
  e^{-(X^2+B_1^2+B_2^2)}
  (X-\ii B_1)^m(X-\ii B_2)^m
  |B_1-B_2|^\alpha
  \,\dd X\,\dd B_1\,\dd B_2.
\end{equation}

\begin{lemma}
For \(\alpha>-1\) and \(|t|<1/2\),
\begin{equation}\label{eq:J-alpha-generating}
  \sum_{m=0}^\infty\frac{J_m^{(\alpha)}}{m!}t^m
  =
  \frac{\sqrt\pi\,Z_\alpha^\circ}
       {(1-t/2)^{1+\alpha/2}},
\end{equation}
where the power is taken with the analytic branch equal to \(1\) at
\(t=0\).  Consequently,
\begin{equation}\label{eq:J-alpha-value}
  J_m^{(\alpha)}
  =
  \sqrt\pi\,Z_\alpha^\circ\,
  2^{-m}\left(1+\frac\alpha2\right)_m.
\end{equation}
\end{lemma}

\begin{proof}
For \(t\) in a compact subset of \(\{|t|<1/2\}\), choose \(r<1/2\) such
that \(|t|\le r\).  Using
\[
  |X-\ii B_1|\,|X-\ii B_2|
  \le
  X^2+\frac{B_1^2+B_2^2}{2},
\]
we have
\[
  \left|e^{t(X-\ii B_1)(X-\ii B_2)}\right|
  \le
  \exp\!\left(r\left[X^2+\frac{B_1^2+B_2^2}{2}\right]\right).
\]
Therefore the absolute value of the integrand below is bounded by
\[
  |B_1-B_2|^\alpha
  e^{-(1-r)X^2-(1-r/2)(B_1^2+B_2^2)},
\]
which is integrable on \(\R^3\) because \(\alpha>-1\).  Dominated
convergence permits summation under the integral sign and gives
\begin{align}
  \sum_{m=0}^\infty\frac{J_m^{(\alpha)}}{m!}t^m
  &=
  \int_{\R^3}
  e^{-(X^2+B_1^2+B_2^2)}|B_1-B_2|^\alpha
  \notag\\
  &\qquad\times
  e^{t(X-\ii B_1)(X-\ii B_2)}
  \,\dd X\,\dd B_1\,\dd B_2.
  \label{eq:J-generating-integral}
\end{align}
Make the orthogonal change of variables
\[
  U=\frac{B_1+B_2}{\sqrt2},
  \qquad
  V=\frac{B_1-B_2}{\sqrt2}.
\]
Then
\[
  |B_1-B_2|^\alpha=2^{\alpha/2}|V|^\alpha
\]
and
\[
  (X-\ii B_1)(X-\ii B_2)
  =
  X^2-\ii\sqrt2\,XU-\frac{U^2-V^2}{2}.
\]
Thus the right-hand side of \eqref{eq:J-generating-integral} becomes
\begin{align}
  2^{\alpha/2}
  \int_{\R^3}|V|^\alpha
  \exp\!\bigg(
  &-(1-t)X^2
   -\left(1+\frac t2\right)U^2 \notag\\
  &-\left(1-\frac t2\right)V^2
   -\ii\sqrt2\,tXU
  \bigg)
  \,\dd X\,\dd U\,\dd V.
  \label{eq:J-after-orthogonal-change}
\end{align}
We first take \(t\in(-1/2,1/2)\).  The standard Gaussian identity
\[
  \int_{\R}e^{-AX^2+bX}\,\dd X
  =
  \sqrt{\frac\pi A}\,e^{b^2/(4A)},
  \qquad A>0,\quad b\in\C,
\]
which follows from the real formula by analytic continuation, gives
\[
  \int_{\R}
  e^{-(1-t)X^2-\ii\sqrt2\,tXU}\,\dd X
  =
  \sqrt{\frac{\pi}{1-t}}
  \exp\!\left(-\frac{t^2}{2(1-t)}U^2\right).
\]
The subsequent \(U\)-integration yields
\begin{equation}\label{eq:XU-Gaussian}
  \int_{\R^2}
  e^{-(1-t)X^2-(1+t/2)U^2-\ii\sqrt2\,tXU}
  \,\dd X\,\dd U
  =
  \frac{\pi}{\sqrt{1-t/2}}.
\end{equation}
The remaining integral is
\begin{equation}\label{eq:V-Gaussian}
  \int_{\R}|V|^\alpha e^{-(1-t/2)V^2}\,\dd V
  =
  \Gamma\!\left(\frac{\alpha+1}{2}\right)
  (1-t/2)^{-(\alpha+1)/2}.
\end{equation}
Combining \eqref{eq:J-after-orthogonal-change},
\eqref{eq:XU-Gaussian}, \eqref{eq:V-Gaussian}, and
\eqref{eq:Z-alpha-base}, we obtain \eqref{eq:J-alpha-generating} for real
\(|t|<1/2\).  The integral in \eqref{eq:J-generating-integral} and the
right-hand side of \eqref{eq:J-alpha-generating} are holomorphic in the disc
\(|t|<1/2\), so the identity theorem gives the formula throughout that disc.
Finally, comparison with the binomial expansion
\[
  (1-s)^{-u}
  =
  \sum_{m=0}^\infty\frac{(u)_m}{m!}s^m,
  \qquad |s|<1,
\]
with \(u=1+\alpha/2\) and \(s=t/2\), gives
\eqref{eq:J-alpha-value}.
\end{proof}

For a polynomial \(M\), write
\[
  \|M\|_N^2
  :=
  \int_{\R}|M(x)|^2e^{-Nx^2/2}\,\dd x.
\]

\begin{theorem}\label{thm:common-L2}
Let \(\beta\in\{1,2\}\), let \(H_N\sim\Pp_{\beta,N}\), and define the
monic polynomial
\[
  R_N(z):=\ii^{-N}P_N(\ii z),
  \qquad z\in\C.
\]
Then
\begin{equation}\label{eq:common-L2}
  \E\|R_N\|_N^2
  =
  \sqrt{2\pi}\,
  \frac{(1+2/\beta)_N}{N^{N+1/2}}
  =
  \kappa_{\beta,N}h_N,
\end{equation}
where
\[
  h_N:=\sqrt{2\pi}\frac{N!}{N^{N+1/2}},
  \qquad
  \kappa_{\beta,N}:=\frac{(1+2/\beta)_N}{N!}.
\]
\end{theorem}

\begin{proof}
By \Cref{lem:real-coefficients}, \(P_N\) has real coefficients.  Therefore,
for every \(x\in\R\),
\[
  |R_N(x)|^2=P_N(\ii x)P_N(-\ii x).
\]
 Using
\Cref{prop:scalar-permanent-duality} with \(\alpha=4/\beta\), 
 \begin{align*}
  \E\|R_N\|_N^2
  &=
  \int_{\R}e^{-Nx^2/2}
  \E\!\left[P_N(\ii x)P_N(-\ii x)\right]\,\dd x\\
  &=
  \frac1{Z_{\alpha,N}}
  \int_{\R^3}
  e^{-\frac N2(x^2+y_1^2+y_2^2)}
  (x-\ii y_1)^N(x-\ii y_2)^N
  |y_1-y_2|^\alpha
  \,\dd x\,\dd y_1\,\dd y_2.
\end{align*}
Make the change of variables
\[
  x=\sqrt{\frac2N}\,X,
  \qquad
  y_j=\sqrt{\frac2N}\,B_j,
  \qquad j=1,2.
\]
Combining this scaling with \eqref{eq:Z-alpha-N-value},
\eqref{eq:J-alpha-m}, and \eqref{eq:J-alpha-value}, we obtain
\[
  \E\|R_N\|_N^2
  =
  \left(\frac2N\right)^{N+1/2}
  \frac{J_N^{(\alpha)}}{Z_\alpha^\circ}
  =
  \sqrt{2\pi}\,
  \frac{(1+\alpha/2)_N}{N^{N+1/2}}.
\]
Since \(\alpha=4/\beta\), this proves \eqref{eq:common-L2}.
\end{proof}

\section{Distribution of zeros}

If \(M\) is a polynomial of exact degree \(N\), with zeros
\(z_1,\dots,z_N\) counted with algebraic multiplicity, define its normalized
zero counting measure by
\[
  \nu_M:=\frac1N\sum_{j=1}^N\delta_{z_j}.
\]

\subsection{Hermite polynomials}

Recall the Gaussian weighted norm
\[
  \|M\|_N^2
  :=
  \int_{\R}|M(x)|^2e^{-Nx^2/2}\,\dd x.
\]
Let \(\mathsf H_n\) denote the  Hermite polynomial,
\[
  \mathsf H_n(y)
  :=
  (-1)^ne^{y^2}\frac{\dd^n}{\dd y^n}e^{-y^2}.
\]
It has leading coefficient \(2^n\) and satisfies
\begin{equation}\label{eq:Hermite-orthogonality}
  \int_{\R}\mathsf H_n(y)\mathsf H_m(y)e^{-y^2}\,\dd y
  =
  \sqrt\pi\,2^n n!\,\delta_{nm}.
\end{equation}
For \(0\le n\le N\), define the scaled monic Hermite polynomial
\begin{equation}\label{eq:scaled-monic-Hermite}
  \pi_{n,N}(x)
  :=
  \frac{\mathsf H_n\!\left(\sqrt{N/2}\,x\right)}
       {2^n(N/2)^{n/2}}.
\end{equation}
The substitution \(y=\sqrt{N/2}\,x\) in
\eqref{eq:Hermite-orthogonality} gives
\begin{equation}\label{eq:scaled-Hermite-norm}
  h_{n,N}
  :=
  \|\pi_{n,N}\|_N^2
  =
  \sqrt{2\pi}\,\frac{n!}{N^{n+1/2}}.
\end{equation}
In particular,
\[
  h_{N,N}=h_N
  =
  \sqrt{2\pi}\,\frac{N!}{N^{N+1/2}},
\]
where \(h_N\) is the quantity appearing in \Cref{thm:common-L2}.
Stirling's formula yields
\begin{equation}\label{eq:Hermite-norm-root}
  \lim_{N\to\infty}h_N^{1/(2N)}=e^{-1/2}.
\end{equation}

\begin{lemma}\label{lem:least-monic-L2}
If \(M\) is any monic complex polynomial of degree \(N\), then
\[
  \|M\|_N^2\ge h_N.
\]
Equality holds if and only if \(M=\pi_{N,N}\).
\end{lemma}

\begin{proof}
Set \(S=M-\pi_{N,N}\).  Then \(\deg S\le N-1\), while
\(\pi_{N,N}\) is orthogonal in
\(L^2(\R,e^{-Nx^2/2}\dd x)\) to every polynomial of degree at most
\(N-1\).  Hence
\[
  \|M\|_N^2
  =
  \|\pi_{N,N}\|_N^2+\|S\|_N^2
  =
  h_N+\|S\|_N^2.
\]
The assertion follows.
\end{proof}

We shall also use Mehler's formula \cite[Example~4.18]{Janson1997}.  Define
\[
  \psi_k(y)
  :=
  \frac{\mathsf H_k(y)e^{-y^2/2}}
       {\bigl(\sqrt\pi\,2^k k!\bigr)^{1/2}}.
\]
For \(-1<r<1\) and \(x,y\in\R\),
\[
  \sum_{k=0}^{\infty}r^k\psi_k(x)\psi_k(y)
  =
  \frac1{\sqrt{\pi(1-r^2)}}
  \exp\!\left(
    -\frac{(1+r^2)(x^2+y^2)-4rxy}{2(1-r^2)}
  \right).
\]
In particular,
\begin{equation}\label{eq:Mehler-diagonal}
  \sum_{k=0}^{\infty}r^k\psi_k(y)^2
  =
  \frac1{\sqrt{\pi(1-r^2)}}
  \exp\!\left(-\frac{1-r}{1+r}y^2\right),
  \qquad 0<r<1.
\end{equation}

Put
\[
  \phi_{k,N}(x)
  :=
  \frac{\pi_{k,N}(x)e^{-Nx^2/4}}{\sqrt{h_{k,N}}},
  \qquad
  K_N(x,x)
  :=
  \sum_{k=0}^N|\phi_{k,N}(x)|^2.
\]
The functions \(\phi_{k,N}\) are orthonormal in \(L^2(\R,\dd x)\).

\begin{lemma}\label{lem:kernel-Nikolskii}
There is a constant \(C>0\), independent of \(N\), such that
\begin{equation}\label{eq:global-kernel-bound}
  \sup_{x\in\R}K_N(x,x)\le CN,
  \qquad N\ge1.
\end{equation}
Consequently, every complex polynomial \(M\) of degree at most \(N\)
satisfies
\begin{equation}\label{eq:Nikolskii-bound}
  \sup_{x\in\R}|M(x)|e^{-Nx^2/4}
  \le
  \sqrt{CN}\,\|M\|_N.
\end{equation}
\end{lemma}

\begin{proof}
From \eqref{eq:scaled-monic-Hermite} and
\eqref{eq:scaled-Hermite-norm},
\[
  \phi_{k,N}(x)
  =
  \left(\frac N2\right)^{1/4}
  \psi_k\!\left(\sqrt{N/2}\,x\right).
\]
Choose \(r=e^{-1/(N+1)}\).  Since \(r^k\ge r^N\ge e^{-1}\) for
\(0\le k\le N\), \eqref{eq:Mehler-diagonal} gives, uniformly in
\(y\in\R\),
\[
  \sum_{k=0}^N\psi_k(y)^2
  \le
  e\sum_{k=0}^{\infty}r^k\psi_k(y)^2
  \le
  \frac{e}{\sqrt{\pi(1-e^{-2/(N+1)})}}
  \le
  C_0\sqrt N.
\]
Here we used \(1-e^{-u}\ge u/2\) for \(0\le u\le1\).  Multiplication by
\((N/2)^{1/2}\) proves \eqref{eq:global-kernel-bound}.

Expand \(M=\sum_{k=0}^Nc_k\pi_{k,N}\).  By Cauchy--Schwarz,
\begin{align*}
  |M(x)|e^{-Nx^2/4}
  &\le
  \left(\sum_{k=0}^N|c_k|^2h_{k,N}\right)^{1/2}
  \left(\sum_{k=0}^N
    \frac{|\pi_{k,N}(x)|^2e^{-Nx^2/2}}{h_{k,N}}
  \right)^{1/2}\\
  &=
  \|M\|_N K_N(x,x)^{1/2}.
\end{align*}
Taking the supremum and using \eqref{eq:global-kernel-bound} proves
\eqref{eq:Nikolskii-bound}.
\end{proof}

\subsection{Distribution of zeros}
Set
\[
  w(x):=e^{-x^2/4}.
\]
For this Gaussian weight, the weighted equilibrium measure is the standard
semicircle law \(\mu_{\mathrm{sc}}\), and the corresponding modified Robin
constant is \(F_w=1/2\).  This is the classical quadratic equilibrium
problem; see \cite[Exercise~2.6.4]{AndersonGuionnetZeitouni2010}.

For a monic complex polynomial \(M\) of degree \(N\), define
\[
  \|w^NM\|_{\R}
  :=
  \sup_{x\in\R}|M(x)|e^{-Nx^2/4},
\]
and let
\begin{equation}\label{eq:weighted-Chebyshev-number}
  t_N(w)
  :=
  \inf\left\{
    \|w^NM\|_{\R}:
    M\text{ is a monic complex polynomial of degree }N
  \right\}.
\end{equation}

\begin{theorem}\label{thm:Saff-Totik-specialized}
For the Gaussian weight \(w(x)=e^{-x^2/4}\), the following statements hold.
\begin{enumerate}[label=\textup{(\alph*)},leftmargin=2em]
\item
\begin{equation}\label{eq:weighted-Chebyshev-limit}
  \lim_{N\to\infty}t_N(w)^{1/N}=e^{-F_w}=e^{-1/2}.
\end{equation}

\item
Let \(M_N\) be a sequence of monic complex polynomials of exact degree
\(N\).  If
\begin{equation}\label{eq:weighted-asymptotic-extremality}
  \lim_{N\to\infty}\|w^NM_N\|_{\R}^{1/N}=e^{-1/2},
\end{equation}
then
\[
  \nu_{M_N}\weak\mu_{\mathrm{sc}}.
\]
\end{enumerate}
\end{theorem}

Part~\textup{(a)} is the weighted Chebyshev theorem
\cite[Chapter~III, Theorem~3.1]{SaffTotik1997}, specialized to the
Gaussian weight \(w\).  Part~\textup{(b)} is the
zero-distribution theorem for weighted asymptotically extremal monic
polynomials \cite[Chapter~III, Theorem~4.2, p.~170]{SaffTotik1997}; see
also \cite[proof of Theorem~2.3]{PritskerXie2015} for the same application
to Freud weights, with the Gaussian weight as a special case.

\begin{theorem}\label{thm:polynomial-excess-criterion}
Let \((M_N)_{N\ge1}\) be a sequence of random monic polynomials defined on
a common probability space, with \(M_N\) of exact degree \(N\).  Suppose
that there are constants \(A>0\) and \(m>0\) such that
\[
  \E\|M_N\|_N^2\le A N^m h_N,
  \qquad N\ge1.
\]
Then
\[
  \nu_{M_N}\weak\mu_{\mathrm{sc}}
  \qquad\text{almost surely.}
\]
\end{theorem}

\begin{proof}
By \Cref{lem:least-monic-L2},
\[
  X_N:=\frac{\|M_N\|_N^2}{h_N}
\]
is a nonnegative random variable satisfying \(X_N\ge1\) and
\(\E X_N\le A N^m\).  Markov's inequality gives
\[
  \Pp\{X_N>N^{m+2}\}\le A N^{-2}.
\]
Therefore
\[
  \sum_{N=1}^{\infty}\Pp\{X_N>N^{m+2}\}<\infty.
\]
By the first Borel--Cantelli lemma, almost surely there is a random
\(N_0\) such that
\[
  \|M_N\|_N^2\le N^{m+2}h_N,
  \qquad N\ge N_0.
\]
On this event, \Cref{lem:kernel-Nikolskii} gives
\[
  \|w^NM_N\|_{\R}
  \le
  \sqrt C\,N^{(m+3)/2}\sqrt{h_N}.
\]
Taking \(N\)-th roots and using \eqref{eq:Hermite-norm-root}, we obtain
\[
  \limsup_{N\to\infty}\|w^NM_N\|_{\R}^{1/N}
  \le
  e^{-1/2}.
\]
On the other hand, by \eqref{eq:weighted-Chebyshev-number},
\(\|w^NM_N\|_{\R}\ge t_N(w)\).  Equation
\eqref{eq:weighted-Chebyshev-limit} therefore gives the reverse liminf.
Thus \eqref{eq:weighted-asymptotic-extremality} holds almost surely, and
part~\textup{(b)} of \Cref{thm:Saff-Totik-specialized} completes the proof.
\end{proof}

Now we are ready to prove \Cref{thm:main}.

\begin{proof}[Proof of \Cref{thm:main}]
Recall that
\(
  R_N(z):=\ii^{-N}P_N(\ii z),
\,\, z\in\C,
\)
is monic of degree \(N\).  By \Cref{thm:common-L2},
\[
  \E\|R_N\|_N^2
  =
  \kappa_{\beta,N}h_N.
\]
From the definition of \(\kappa_{\beta,N}\),
\[
  \kappa_{1,N}=\binom{N+2}{2},
  \qquad
  \kappa_{2,N}=N+1.
\]
Hence \(\kappa_{\beta,N}\le9N^2\) for \(N\ge1\).  Applying
\Cref{thm:polynomial-excess-criterion} with \(A=9\) and \(m=2\) gives
\[
  \nu_{R_N}\weak\mu_{\mathrm{sc}}
  \qquad\text{almost surely.}
\]
If \(r_{1,N},\dots,r_{N,N}\) are the zeros of \(R_N\), counted with
multiplicity, then \(\ii r_{1,N},\dots,\ii r_{N,N}\) are the zeros of
\(P_N\).  Therefore, for every bounded continuous function
\(f:\C\to\C\),
\begin{align*}
  \int_{\C}f(z)\,\dd\nu_{P_N}(z)
  &=
  \int_{\C}f(\ii z)\,\dd\nu_{R_N}(z)\\
  &\to
  \int_{\R}f(\ii x)\,\dd\mu_{\mathrm{sc}}(x)\\
  &=
  \int_{\C}f(z)\,
  \dd\bigl((\mathsf m_i)_\#\mu_{\mathrm{sc}}\bigr)(z)
\end{align*}
almost surely.  This is the asserted weak convergence.  
\end{proof}

\end{document}